\documentclass[12pt, reqno]{amsart}
\usepackage{amsmath, amsthm, amscd, amsfonts, amssymb, graphicx, color}
\usepackage{tikz-cd}

\newtheorem{theorem}{Theorem}[section]
\newtheorem{lemma}[theorem]{Lemma}

\newtheorem{corollary}[theorem]{Corollary}
\theoremstyle{definition}
\newtheorem{definition}[theorem]{Definition}
\newtheorem{example}[theorem]{Example}

\newtheorem{question}[theorem]{Question}

\theoremstyle{remark}
\newtheorem{remark}[theorem]{Remark}

\begin{document}
\setcounter{page}{1}

\title{On $a$-locally Closed Sets}

\author{B. İZCİ $^{\rm a}$, M. ÖZKOÇ $^{\rm b,\ast}$}

\address{$^{1}$Department of Mathematics, Graduate School of Natural and Applied Sciences, Muğla Sıtkı Koçman University, 48000,  Menteşe-Muğla, Turkey.}
\email{bilgeizci@posta.mu.edu.tr \& bilgeizci@hotmail.com}

\address{$^{2}$Department of Mathematics, Faculty of Sciences, Muğla Sıtkı Koçman University, 48000, Menteşe-Muğla, Turkey.}
 \email{murad.ozkoc@mu.edu.tr \& murad.ozkoc@gmail.com}




\date{$^{*}$Corresponding author}

\begin{abstract}
The aim of this paper is to introduce the notion of $a$-locally closed set by utilizing $a$-open sets defined by Ekici and to study some properties of this new notion. Also, some characterizations and many fundamental results regarding this new concept are obtained. Moreover, the relationships between the concepts defined within the scope of this study and some other types of local closed sets in the literature have been revealed. 
\end{abstract} 
\maketitle

\section{Introduction}
The elements of a topology defined as a family of sets consisting of some subsets of a non-empty set $X$ and closed under finite intersection and any union operation are called open sets. The concept of open set has an important place in general topology and is one of the focal points of research for many mathematicians all over the world. The study of different versions of continuity, separation axioms, compactness, connectedness and other concepts defined with the help of special and general forms of the open set concept are important topics of study in general topology. Starting in 1963 with Levine's introduction of the notion of semiopen set, the process continued with Njastad's study of $\alpha$-open set in 1965 and Ekici's study of $e$-open set in 2008. These works of Levine, Njastad and Ekici inspired the work done today and since then, different types of open sets have been intensively studied. 

To generalize different ideas in topology, many mathematicians have focused on various forms of open sets such as $\alpha$-open set, semi-open set, pre-open set, $b$-open set, $\beta$-open set, $e$-open set and $e^*$-open set. The weak and strong forms of these concepts have been studied by many researchers. These studies have evolved over time into the concept of local closed set and some forms of this concept have been studied over time. Some of these include the intersection of an open set and a closed set.

\section{Preliminaries}

Throughout this paper, unless otherwise stated the terms $X$ and $Y$ refer to topological spaces on which no separation axioms are imposed. For a subset $A$ of $X,$ $cl(A)$ and $int(A)$ stand for the closure of $A$ and the interior of $A$ in $X$, respectively.  $O(X,x)$ stands for the family of all open subsets of $X$ that contain $x$. A subset $A$ is said to be regular open (resp. regular closed) if $A = int(cl(A))$ $($resp. $A=cl(int(A))).$ The $\delta$-interior of a subset $A$ of $X$ is the union of all regular open sets of $X$ contained in $A$ and is denoted by $\delta \text{-}int(A)$. The subset $A$ of a space $X$ is called $\delta$-open if $A=\delta\text{-}int(A)$, i.e., a set is $\delta \text{-}$open if it is the union of some regular open sets.
\\
\begin{definition} 
A subset $A$ of a space $X$ is called:\\
   $a)$ semi-open \cite{11} if $A \subseteq cl(int(A))$;
   \\
   $b)$ $\alpha$-open \cite{14} if $A \subseteq int(cl(int(A)))$;
   \\
   $c)$ $a$-open \cite{6} if $A \subseteq int(cl(\delta\text{-}int(A)))$;
   \\
   $d)$ $b$-open \cite{3} if $A \subseteq cl(int(A))\cup int(cl(A))$;
   \\
   $e)$ $e$-open \cite{8} if $A\subseteq cl(\delta\text{-}int(A))\cup int(\delta\text{-}cl(A))$;
   \\
   $f)$ feebly open \cite{13} if there exists an open set $U$ such that $ U \subseteq A \subseteq scl(U),$ where $scl(U)$ denotes the semi-closure of $U$.
\end{definition}

The family of all semiopen (resp. $\alpha$-open, $a$-open, $b$-open, $e$-open, feebly open) sets in $X$ is denoted by $SO(X)$  (resp. $\alpha O(X), aO(X), BO(X), eO(X), FO(X)).$
The complement of a semi-open (resp. $\alpha$-open, $a$-open, $b$-open, $e$-open, feebly open) set is said to be semi-closed (resp. $\alpha$-closed, $a$-closed, $b$-closed, $e$-closed, feebly closed). The family of all semiclosed (resp. $\alpha$-closed, $a$-closed, $b$-closed, $e$-closed, feebly closed) sets in $X$ is denoted by $SC(X)$  (resp. $\alpha C(X),$ $aC(X),$ $BC(X),$ $eC(X),$ $FC(X)).$

 \begin{definition}
    The semi-closure (resp. $a$-closure) of a subset $A$ of a space is the smallest semi-closed (resp. $a$-closed) set containing $A$ and denoted by $scl(A)$ (resp. $a\text{-}cl(A)).$ Dually, the semi-interior (resp. $a$-interior) of a subset $A$ of a space $X$ is the largest semi-open (resp. $a$-open) set contained in $A$ and denoted by $sint(A)$ (resp. $a\text{-}int(A)).$
\end{definition}

\begin{definition}
A subset $A$ of a space $(X,\tau)$ is called: 
\\
$a)$ locally closed \cite{1} if $A=U \cap V,$ where $U$ is open and $V$ is closed in $X;$
\\
$b)$ $\alpha $-locally closed \cite{10} if $A=U \cap V,$ where $U$ is $\alpha$-open and $V$ is $\alpha$-closed in $X;$
\\
$c)$ $b$-locally closed \cite{10} if $A=U \cap V,$ where $U$ is $b$-open and $V$ is $b$-closed in $X;$
\\
$d)$ $e$-locally closed \cite{101} if $A=U \cap V,$ where $U$ is $e$-open and $V$ is $e$-closed in $X;$
\\
$e)$ feebly locally closed \cite{13} if $A=U \cap V,$ where $U$ is feebly open and $V$ is feebly closed in $X.$
\end{definition}

\begin{lemma} \cite{6} \label{aopen}
 Let $(X,\tau)$ be a topological space. Then, the following hold:
\\
$a)$ $\emptyset,X\in aO(X),$ \\
$b)$ If $A,B\in aO(X),$ then $A\cap B\in aO(X),$ \\
$c)$ If $\mathcal{A}\subseteq aO(X),$ then $\bigcup \mathcal{A}\in aO(X).$
\end{lemma}

\begin{corollary}\label{aopen1}
 Let $(X,\tau)$ be a topological space. Then, the following hold:
 \\
$a)$ $\emptyset,X\in aC(X),$ \\
$b)$ If $A,B\in aC(X),$ then $A\cup B\in aC(X),$ \\
$c)$ If $\mathcal{A}\subseteq aC(X),$ then $\bigcap \mathcal{A}\in aC(X).$
\end{corollary}

\begin{definition}
A subset $A$ of a space $X$ is called dense if $cl(A)=X$. A space $X$ is called submaximal if every dense subset of $X$ is open in $X.$
 \end{definition}
 
\begin{definition} \cite{101}
A subset $A$ of a space $X$ is called dense if $e$-$cl(A)=X$. A space $X$ is called $e$-submaximal if every $e$-dense subset of $X$ is $e$-open in $X.$
\end{definition}

\section{$a$-locally closed sets}

\begin{definition}
A subset $A$ of a topological space $X$ is called $a$-locally closed if it is the intersection of an $a$-open and an $a$-closed set. The complement of an $a$-locally closed set is called $a$-locally open. The family of all $a$-locally closed sets (resp. $a$-locally open) in a space $X$ will be denoted by $aLC(X) \ (\text{resp. } aLO(X)).$ 
\end{definition}
	
	\begin{theorem}
	Let $X$ be a topological space. Then, the following hold.\\
	$a)$ $aO(X)\subseteq aLC(X),$\\
	$b)$ $aC(X)\subseteq aLC(X).$
	\end{theorem}
	\begin{proof}
	$(a)$  Let $A\in aO(X).$ \\
$\left.\begin{array}{rr}  
A \in aO(X) \\ (U:=A)(V:=X) \end{array}\right\}\Rightarrow (U \in aO(X))(V \in aC(X))(A=U \cap V)$
\\
$\begin{array}{l}\Rightarrow A\in aLC(X).\end{array}$
\\

$(b)$ Let $A\in aC(X).$ \\
	 $\left.\begin{array}{rr} A \in aC(X) \\ (U:=X)(V:=A) \end{array}\right\}\Rightarrow (U \in aO(X))(V \in aC(X))(A=U\cap V)$
  \\ 
  $\begin{array}{l}\Rightarrow A\in aLC(X).\end{array}$	
	 \end{proof}

\begin{theorem}
   Let $X$ be a topological space. Then, $aLC(X)\subseteq eLC(X).$
    \end{theorem}
    
\begin{proof}
Let $A\in aLC(X).$
\\
$\left.\begin{array}{rr} A\in aLC(X)\Rightarrow (\exists U\in aO(X))(\exists V\in aC(X))(A=U\cap V)\\ (aO(X)\subseteq eO(X))(aC(X)\subseteq eC(X))\end{array}\right\}\Rightarrow$\\
$\begin{array}{l} \Rightarrow (\exists U\in eO(X))(\exists V\in eC(X))(A=U\cap V)\end{array}$
\\
$\begin{array}{l} \Rightarrow A\in eLC(X).
\end{array}$
\end{proof}

	\begin{lemma}\label{lemma2}
	Let $X$ be a topological space. Then, the following hold.\\
    $a)$ $aO(X) \subseteq FO(X),$\\
    $b)$ $aC(X) \subseteq FC(X).$
    \end{lemma}
	
\begin{proof}
    $a)$ Let $A\in aO(X).$\\
    $\left.\begin{array}{rr}A\in aO(X)\Rightarrow A\subseteq int(cl(\delta\text{-}int(A))) \\ U:=\delta\text{-}int(A)\end{array}\right\}\Rightarrow $
    \\
    $\begin{array}{l}\Rightarrow (U\in \tau)(U\subseteq A\subseteq int(cl(U))\subseteq U\cup int(cl(U))=scl(U))\end{array}$
    \\
    $\begin{array}{l}\Rightarrow A\in FO(X).\end{array}$
    
\hspace{1cm}   $b)$ Let $A\in aC(X).$\\
	$\begin{array}{l}A\in aC(X)\Rightarrow X\setminus A \in aO(X)\overset{(a)}{\Rightarrow} X \setminus A \in FO(X)\Rightarrow A\in FC(X).\end{array}$
	\end{proof}

\begin{theorem}
Let $X$ be a topological space. Then, the following statements hold:
\\
$a)$ $aLC(X)\subseteq FLC(X),$
\\
$b)$ $aLC(X)\subseteq \alpha C(X).$
\end{theorem}
    
\begin{proof}
$a)$ Let $A\in aLC(X).$
   \\
   $\left.\begin{array}{rr} A\in aLC(X) \Rightarrow (\exists U\in aO(X))(\exists V\in aC(X))(A=U\cap V) \\   \text{Lemma } \ref{lemma2} \end{array}\right\}\Rightarrow $
\\
$\begin{array}{l}\Rightarrow (\exists U\in FO(X))(\exists V\in FC(X))(A=U\cap V)\end{array}$
\\
$\begin{array}{l}
\Rightarrow  A\in FLC(X).
\end{array}$

$b)$ Let $A\in aLC(X).$ \\
$\left.\begin{array}{rr}
   A\in aLC(X) \Rightarrow  (\exists U\in aO(X))(\exists V\in aC(X))(A=U\cap V) \\  (aO(X)\subseteq \alpha O(X))(aC(X)\subseteq \alpha C(X))
   \end{array}\right\}\Rightarrow$
   \\
  $\begin{array}{l} \Rightarrow (\exists U\in \alpha O(X))(\exists V\in \alpha C(X))(A=U\cap V)\end{array}$
    \\
  $\begin{array}{l} \Rightarrow A\in \alpha LC(X).
  \end{array}$
    \end{proof}

\begin{remark}
	We have the following diagram from the previous definitions and results given above.
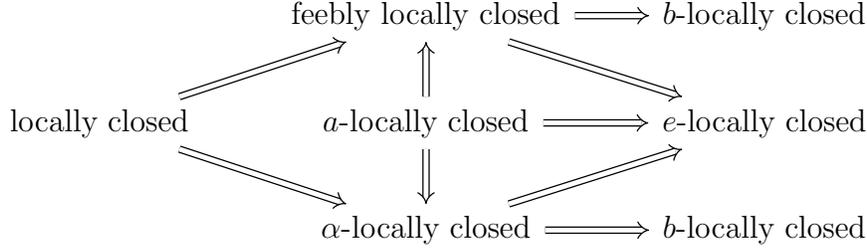
\begin{equation*}
\begin{tikzcd}
 & \text{feebly locally closed} \arrow[r, Rightarrow] \arrow[rd, Rightarrow]                      & b\text{-locally closed}\\
\text{locally closed}  \arrow[rd, Rightarrow] \arrow[ru, Rightarrow] & a\text{-locally closed} \arrow[r, Rightarrow] \arrow[u, Rightarrow] \arrow[d, Rightarrow]                     & e\text{-locally closed} \\
\text{ } & \alpha\text{-locally closed}  \arrow[ru, Rightarrow] \arrow[r, Rightarrow] & b\text{-locally closed}
\end{tikzcd}
\end{equation*}
\end{remark}

The converses given above implications need not to be true as shown by the following examples.

\begin{example}
Let $X=\{a,b,c,d\}$ and $\tau=\{\emptyset,X,\{a\},\{b\},\{a,b\},\{a,c,d\}\}$. Simple calculations show that $eLC(X)=\alpha LC(X)=FLC(X)=2^X$ and $aLC(X)=\{\emptyset,X,\{b\},\{a,c,d\}\}$. Then, it is clear that the set $\{a\}$ is feebly locally closed and so $e$-locally closed. Also, it is $\alpha$-locally closed but not $a$-locally closed.
    \end{example}

 \begin{example}
    Let $X=\{a,b,c,d\}$ and $\tau=\{\emptyset, X,\{a\},\{b\},\{a,b\},\{a,c\},\{a,b,c\},\\ \{a,b,d\}\}$. Simple calculations show that $LC(X)=2^X$, $aLC(X)=\{\emptyset, X, \{b\}, \{d\}, \\ \{a,c\}, \{b,d\}, \{a,b,c\}, \{a,c,d\}\}$. Then, it is clear that the set $\{a\}$ is locally closed but not $a$-locally closed.
  \end{example}

\begin{example}
Let $X=\{a,b,c,d\}$ ve $\tau=\{\emptyset, X,\{a\},\{b\},\{a,b\}\}$. Simple calculations show that $LC(X)=\{\emptyset, X,\{a\},\{b\},\{a,b\},\{c,d\},\{a,c,d\},\{b,c,d\}\}$ and $aLC(X)=2^X.$ Then, it is clear that the set $\{c\}$ is an $a$-locally closed set but it is not locally closed.
\end{example}

\begin{question}
Are the notions $b$-locally closedness and $e$-locally closedness independent? 
\end{question}

\begin{theorem}
Let $A$ and $B$ be two subsets of a space $X.$ 
If A and B are $a$-locally closed, then so are their intersections.
	      \end{theorem}
	      \begin{proof}
	      Let $A,B \in aLC(X).$ 
	      \\
	      $\left.\begin{array}{rr} A \in aLC(X) \Rightarrow (\exists U_1 \in aO(X))(\exists V_1 \in aC(X))(A= U_1 \cap V_1) \\ B \in aLC(X) \Rightarrow (\exists U_2 \in aO(X))(\exists V_2 \in aC(X))(B= U_2 \cap V_2) \end{array}\right\}\overset{\text{Lemma } \ref{aopen} }{\Rightarrow}$ 
	      \\
	      $\begin{array}{l}
	      \Rightarrow (U_1 \cap U_2 \in aO(X))(V_1 \cap V_2 \in aC(X))(A \cap B = (U_1 \cap U_2) \cap (V_1 \cap V_2))
	      \end{array}$
	      \\
	      $\begin{array}{l}
	      \Rightarrow A \cap B \in aLC(X).
	      \end{array}$
	      \end{proof}

\begin{lemma}\label{aek}
Let $A$ and $B$ be two subsets of a space $X.$ Then, the following hold:\\
$a)$ If $A\in aO(X)$ and $B\in eO(X),$ then $A\cap B\in eO(X),$ \\
$b)$ If $A\in aC(X)$ and $B\in eC(X),$ then $A\cap B\in eC(X).$ 
\end{lemma}

\begin{proof}
$\mathbf{(a)}$ Let $A\in aO(X)$ and $B\in eO(X).$\\
$\left.\begin{array}{rr}  
A \in aO(X) \Rightarrow A \subseteq int(cl(\delta \text{-}int(A)))\\B \in eO(X) \Rightarrow B \subseteq int(\delta \text{-}cl(B)) \cup cl(\delta\text{-}int(B))  \end{array}\right\}{\Rightarrow} $
\\
$\begin{array}{rcl} \Rightarrow A \cap B & \subseteq & [int(cl(\delta \text{-}int(A))) \cap int(\delta \text{-}cl(B))] \cup [int(cl(\delta \text{-}int(A))) \cap cl(\delta \text{-}int(B))] 
\\
\\
& \subseteq & int[cl(\delta \text{-}int(A)) \cap int(\delta \text{-}cl(B))] \cup cl[int(cl(\delta \text{-}int(A))) \cap \delta \text{-}int(B)]
\\
\\
& \subseteq & int(cl[int(\delta \text{-}cl(B)) \cap \delta \text{-}int(A)]) \cup cl[\delta\text{-}int(\delta\text{-}cl(\delta \text{-}int(A))) \cap \delta \text{-}int(\delta \text{-}int(B))]
\\
\\
& \subseteq & int(cl[\delta \text{-}int(\delta \text{-}cl(B)) \cap \delta\text{-} int(\delta \text{-}int(A))]) \cup cl[\delta\text{-}int(\delta\text{-}cl(\delta \text{-}int(A))) \cap \delta \text{-}int(B)]
\\
\\
& \subseteq & int(cl[\delta \text{-}int[\delta \text{-}cl(B) \cap \delta \text{-}int(A)]]) \cup cl(\delta\text{-}int[\delta\text{-}cl(\delta \text{-}int(A) \cap \delta \text{-}int(B))])
\\
\\
& \subseteq & int(cl[\delta \text{-}int[\delta \text{-}cl(\delta \text{-} int(A)\cap B)]]) \cup \delta \text{-} cl(\delta\text{-}int(\delta\text{-}cl(\delta \text{-}int(A \cap B))))
\\
\\
& \subseteq & \delta \text {-}int (\delta\text{-}cl(\delta\text{-}int(\delta\text{-}cl(A \cap B)))) \cup \delta\text{-}cl(\delta\text{-}int(A\cap B))
\\
\\
& = & \delta \text{-}int(\delta \text{-}cl(A \cap B)) \cup cl(\delta \text{-}int(A \cap B))
\\
\\
& = & int(\delta \text{-}cl(A \cap B)) \cap cl(\delta\text{-}int(A \cap B))\end{array}$
\\

 This means $A\cap B\in eO(X).$\\

$b)$ Let $A\in aC(X)$ and $B\in eC(X).$ \\ $\left.\begin{array}{rr}  
A \in aC(X) \Rightarrow X\setminus A \in aO(X) \\ B \in eC(X) \Rightarrow X\setminus B \in eO(X)  \end{array}\right\}\overset{(a)} {\Rightarrow} X \setminus (A \cup B)=(X\setminus A) \cap (X\setminus B)  \in eO(X)
\\
 \Rightarrow A \cup B \in eC(X).$
\end{proof} 
    
    \begin{theorem}
Let $A$ and $B$ be two subsets of a space $X.$ If $A\in aLC(X)$ and $B\in eLC(X),$ then $A\cap B\in eLC(X).$ 
    \end{theorem}
    
    \begin{proof}
     Let $A\in aLC(X)$ and $B\in eLC(X).$
    \\
    $\left.\begin{array}{rr}  
A \in aLC(X) \Rightarrow (\exists U_1 \in aO(X))(\exists V_1 \in aC(X))(A=U_1 \cap V_1) \\B \in eLC(X) \Rightarrow (\exists U_2\in eO(X))(\exists V_2 \in eC(X))(B= U_2 \cap V_2)  \end{array}\right\}\overset{\text{Lemma } \ref{aek}}{\Rightarrow} $
\\
$\left.\begin{array}{rr}
\Rightarrow (U_1 \cap U_2 \in eO(X))(V_1 \cap V_2 \in eC(X))(A \cap B=(U_1 \cap U_2)\cap (V_1 \cap V_2)) \\ (U:=U_1\cap U_2)(V:=V_1\cap V_2)\end{array}\right\}\Rightarrow$
    \\
    $\begin{array}{l}
    \Rightarrow (U\in eO(X))(V \in eC(X))(A \cap B=U\cap V)
    \end{array}$
    \\
    $\begin{array}{l}
    \Rightarrow A\cap B \in eLC(X).
    \end{array}$
    \end{proof}

\begin{theorem}\label{t3}
Let $A$ be a subset of a space $X.$ If $A$ is an $a$-locally closed in $X$, then there exists an $a$-closed set $F$ in $X$ such that $A\cap F= \emptyset$.
\end{theorem}
	
\begin{proof}
Let $A \in aLC(X).$
\\
$\left.\begin{array}{rr}
 A \in aLC(X) \Rightarrow (\exists U \in aO(X))(\exists V \in aC(X))(A=U \cap V) \\ F:= V\setminus U \end{array}\right\}\Rightarrow $
\\
$\begin{array}{l}
\Rightarrow (F \in aC(X))(A \cap F= \emptyset).
\end{array}$
\end{proof}

\begin{remark}
As seen in the example below, 
the converse of the conditional statement given in Theorem \ref{t3} need not always to be true.
 \end{remark}	

\begin{example}
Let $X=\{a,b,c,d\},$ $\tau=\{\emptyset,X,\{a\},\{b\},\{a,b\},\{a,c,d\}\}$ and $A=\{a\}.$ Simple calculations show that
$aO(X) = aC(X) = aLC(X) = \{\emptyset,X,\{b\},\\ \{a,c,d\}\}.$ Then, it is clear that the set $F = \{b\}\in aC(X)$ and $A\cap F=\{a\}\cap \{b\}=\emptyset$ but $A$ is not an $a$-locally closed in $X.$  
\end{example}
 
\begin{theorem}
Let $P$ and $Q$ be two subsets of a space $X$ such that $P\in aO(X)$ and $Q\in aC(X).$ Then, there exist an $a$-open set $E$ and an $a$-closed $F$ such that $P \cap Q \subseteq F$ and $E \subseteq P \cup Q.$ 
\end{theorem}

\begin{proof}
Let $P\in aO(X)$ and $Q\in aC(X).$
\\
     $\left.\begin{array}{rr}  
 (P \in aO(X))(Q \in aC(X))\\ (E:=P \cup a\text{-}int(Q))(F:=Q\cap a\text{-}cl(P))\end{array}\right\}\Rightarrow$
 \\
 $\begin{array}{l}
 \Rightarrow (E \in aO(X))(F \in aC(X))(P \cap Q \subseteq F)(E\subseteq P \cup Q).
 \end{array}$
\end{proof}

	     \begin{theorem}
	        Let $A$ and $B$ be two subsets of a space $X.$ If $A,B \in aLC(X)$, then $A \cap B \in aLC(X)$.  
	     \end{theorem}

\begin{proof}
Let $A,B\in aLC(X).$\\
$\left.\begin{array}{l}
A\in aLC(X)\Rightarrow (\exists U_1\in aO(X))(\exists V_1\in aC(X))(A=U_1\cap V_1) \\ B\in aLC(X)\Rightarrow (\exists U_2\in aO(X))(\exists V_2\in aC(X))(B=U_2\cap V_2)
\end{array}\right\}\overset{\text{Lemma } \ref{aopen} }{\Rightarrow}$
\\
$\begin{array}{l}
\Rightarrow (U_1\cap U_2\in aO(X))(V_1\cap V_2\in aC(X))(A\cap B=(U_1\cap V_1)\cap (U_2\cap V_2))   
\end{array}$
\\
$\begin{array}{l}
\Rightarrow (U_1\cap U_2\in aO(X))(V_1\cap V_2\in aC(X))(A\cap B=(U_1\cap U_2)\cap (V_1\cap V_2))   \end{array}$
\\
$\begin{array}{l}
\Rightarrow A\cap B\in aLC(X).  \end{array}$
\end{proof}

\begin{definition}
 A space $X$ is called an $a$-space if $\tau=\tau^a,$ where $\tau^a=aO(X).$ 
\end{definition}

\begin{example} 
Let $X=\{a,b,c,d\}$ and $\tau=\{\emptyset,X,\{a\},\{b,c,d\}\}.$ Simple calculations show that $\begin{array}{l} \tau^a = \{\emptyset,X,\{a\},\{b,c,d\}\}.\end{array}$ This means that $(X,\tau)$ is an $a$-space since $\tau=\tau^a.$ 
\end{example}

	     \begin{theorem}\label{elck}
      Let $A$ be a subset of a space $X.$ Then, the following statements are equivalent:\\
	     $a)$ $A$ is $a$-locally closed;\\
	     $b)$ $A=P\cap a\text{-}cl(A)$ for some $a$-open set $P$;\\
	     $c)$ $a\text{-}cl(A)\setminus A$ is $a$-closed;\\
	    $d)$ $A\cup (X\setminus a\text{-}cl(A))$ is $a$-open; \\
	     $e)$ $A\subseteq a\text{-}int(A\cup (X\setminus a\text{-}cl(A)))$.
	     \end{theorem}
	     
	     \begin{proof}
	        $(a) \Rightarrow (b):$  Let $A \in aLC(X).$ 
         \\
	        $\begin{array}{l}A \in aLC(X) \Rightarrow (\exists P \in aO(X))(\exists Q \in aC(X))(A=P\cap Q)\end{array}$
         \\
	        $\begin{array}{l}\Rightarrow (\exists P \in aO(X))(\exists Q \in aC(X))(A \subseteq Q)(A=P\cap Q)\end{array}$
         \\
	        $\begin{array}{l}\Rightarrow(\exists P \in aO(X))(a\text{-}cl(A) \subseteq a\text{-}cl(Q)=Q)(A=P\cap Q)\end{array}$
         \\
	        $\begin{array}{l}\Rightarrow(\exists P \in aO(X))(A=A\cap a\text{-}cl(A)=(P\cap Q)\cap a\text{-}cl(A)=P\cap a\text{-}cl(A) \subseteq P \cap Q =A)\end{array}$
         \\
	        $\begin{array}{l}\Rightarrow(\exists P \in aO(X))(A=P \cap a\text{-}cl(A)).
         \end{array}$
\\

 $(b) \Rightarrow (c):$ Let $A\subseteq X.$
	        \\
	        $\left.\begin{array}{rr}
	        A\subseteq X \\ \text{Hypothesis}        \end{array}\right\}\Rightarrow (\exists P\in aO(X))(a\text{-}cl(A)\setminus A=a\text{-}cl(A)\setminus (P \cap a\text{-}cl(A)))$
	        \\
	        $\begin{array}{l}
	        \Rightarrow (X\setminus P\in aC(X))(a\text{-}cl(A)\setminus A=a\text{-}cl(A)\setminus P=a\text{-}cl(A) \cap (X\setminus P))
	        \end{array}$
	        \\
	        $\begin{array}{l}
	        \Rightarrow a\text{-}cl(A)\setminus A\in aC(X).
	        \end{array}$
\\

$(c) \Rightarrow (d):$ Let $A\subseteq X.$
	        \\	
	       $\left.\begin{array}{rr} A\subseteq X\Rightarrow A \cup (X\setminus a\text{-}cl(A))= X\setminus (a\text{-}cl(A)\setminus A) \\ \text{Hypothesis} \end{array}\right\}\Rightarrow A \cup (X\setminus a\text{-}cl(A))\in aO(X).$
\\

$(d) \Rightarrow (e):$ Let $A\subseteq X.$
	        \\
	        $\left.\begin{array}{rr}
	        A \subseteq X \\ \text{Hypothesis}        \end{array}\right\}\Rightarrow A \cup (X\setminus a\text{-}cl(A))\in aO(X)$
	        \\
	        $\left.\begin{array}{rr}
	        \Rightarrow A \cup (X\setminus a\text{-}cl(A))=a\text{-}int(A \cup (X\setminus a\text{-}cl(A))) \\ A\subseteq X\Rightarrow A\subseteq A \cup (X\setminus a\text{-}cl(A))      \end{array}\right\}\Rightarrow A\subseteq a\text{-}int(A \cup (X\setminus a\text{-}cl(A))).$
	       \\
	       
$(e) \Rightarrow (a):$
	       Let $A\subseteq X.$ \\
    $\left.\begin{array}{rr} A\subseteq X \\ \text{Hypothesis}\end{array}\right\}\Rightarrow A \subseteq a \text{-}int(A \cup (X \setminus a \text{-}cl(A)))$
\\
$\begin{array}{rcl} \Rightarrow A=A\cap a\text{-}cl(A) & \subseteq & a \text{-}int(A \cup (X \setminus a \text{-}cl(A)))\cap a\text{-}cl(A) \\ & \subseteq & [A \cup (X \setminus a \text{-}cl(A))]\cap a\text{-}cl(A) \\ & = & [A\cap a \text{-}cl(A)] \cup [(X \setminus a \text{-}cl(A))\cap a\text{-}cl(A)] \\ & = & A \cup \emptyset \\ & = & A  \end{array}$
\\
$\left.\begin{array}{rr} \Rightarrow A =a\text{-}int(A \cup (X \setminus a\text{-}cl(A)))  \cap a\text{-}cl(A) \\ (U:=a\text{-}int(A \cup (X \setminus a\text{-}cl(A))))(V:=a\text{-}cl(A))\end{array}\right\}\Rightarrow $
        \\
        $\begin{array}{l}\Rightarrow (U\in aO(X))(V\in aC(X))(A=U\cap V)\end{array}$
        \\
        $\begin{array}{l}\Rightarrow A\in aLC(X).\end{array}$
	     \end{proof}
	     
     \begin{corollary}\label{alo}
      Let $A$ be a subset of a space $X.$ Then, the following statements are equivalent:\\
	     $a)$ $A$ is $a$-locally open;\\
	     $b)$ $A=Q\cup a\text{-}int(A)$ for some $a$-closed set $Q$;\\
	     $c)$ $(\setminus A) \cup a\text{-}int(A)$ is $a$-closed;\\
	    $d)$ $A\cap (X\setminus a\text{-}int(A))$ is $a$-closed; \\
	     $e)$ $A\supseteq a\text{-}cl(A\cap (X\setminus a\text{-}int(A)))$.
	     \end{corollary}	     
    
	     \begin{theorem}
       If $W \subseteq H \subseteq X$ and $H \in aLC(X)$, then there exists an $a$-locally closed set K such that $W \subseteq K \subseteq H$.
	     \end{theorem}
	     
	     \begin{proof}
	    Let $H \in aLC(X).$
	    \\
       $\left.\begin{array}{rr}
       H \in aLC(X) \overset{\text{Theorem \ref{elck}}}\Rightarrow (\exists P \in aO(X))(H= P \cap a\text{-}cl(H)) \\ W \subseteq H \end{array}\right\}\Rightarrow$
       \\
       $\begin{array}{l}\Rightarrow (W \subseteq P)(P \in aO(X))(P \cap a\text{-}cl(W) \subseteq P \cap a\text{-}cl(H)=H)
       \end{array}$
       \\
	   $\left.\begin{array}{rr} \Rightarrow (P \in aO(X))(W \subseteq P \cap a\text{-}cl(W) \subseteq P \cap a\text{-}cl(H)=H) \\ K:=P\cap a\text{-}cl(W) \end{array}\right\}\Rightarrow$
	   \\ 
	   $\begin{array}{l}
	   \Rightarrow (K \in aLC(X))(W \subseteq K \subseteq H).
	   \end{array}$
	   \end{proof}
	     
	     \begin{definition}
	     A subset $A$ of a space $X$ is called $a$-dense if $a$-$cl(A)=X$. A space $X$ is called $a$-submaximal if every $a$-dense subset of $X$ is $a$-open in $X.$
	     \end{definition}

	     \begin{theorem}
       A topological space $X$ is $a$-submaximal iff $aLC(X)=2^X.$
	     \end{theorem}
	     
	     \begin{proof}
	     $(\Rightarrow):$ 
       Let $(X,\tau)$ be an $a$-submaximal space. Then, it is obvious that $aLC(X) \subseteq 2^X\ldots (1)$\\

Now, let $A \in 2^X.$
	      \\ 
       $\left.\begin{array}{rr} A \in 2^X \\ B:= A \cup (X\setminus a\text{-}cl(A)) \end{array}\right\}\Rightarrow$
       \\ $\begin{array}{l}\Rightarrow X\supseteq a\text{-}cl(B)=a\text{-}cl(A \cup (X\setminus a\text{-}cl(A))) \supseteq a\text{-}cl(A) \cup a\text{-}cl(X\setminus a\text{-}cl(A))=X\end{array}$
	      \\
	      $\left.\begin{array}{rr}
	      \Rightarrow a\text{-}cl(B)=X \\
	      (X,\tau) \text{ is } a\text{-submaximal}
	      \end{array}\right\}\Rightarrow B=A \cup (X\setminus  a\text{-}cl(A)) \in aO(X)$

       $\begin{array}{l}\overset{\text{Theorem \ref{elck}}(d)}\Rightarrow A \in aLC(X)\end{array}$

      Therefore, $2^X\subseteq aLC(X)\ldots (2)$

       $(1),(2)\Rightarrow aLC(X)=2^X.$  
\\

$(\Leftarrow):$ Let $a\text{-}cl(A)=X$. Our aim is to show that $A\in aO(X).$
      \\
$\left.\begin{array}{rr} a\text{-}cl(A)=X \Rightarrow A=A\cup (X\setminus a\text{-}cl(A))  \\ aLC(X)=2^X \end{array}\right\}\overset{\text{Theorem }\ref{elck}(d)}{\Rightarrow} A \in aO(X)$. 
	     \end{proof}

\begin{definition}
A space $(X,\tau)$ is called an $e$-space if $\tau=\tau^e,$ where $\tau^e=eO(X).$    
\end{definition}
    
\begin{remark} If $(X ,\tau)$ is regular and $e$-space, then the notions submaximal, $a$-submaximal, $e$-submaximal coincides with one another.
\end{remark}

\begin{proof}
$(a) \Rightarrow (b):$ Let $A \subseteq X$ and $a\text{-}cl(A)= X$.
\\
$\left.\begin{array}{r} A\subseteq X \\ (X,\tau)\text{ is regular} \end{array} \right\}\Rightarrow \begin{array}{rr} \\ \!\!\!\!\! \left. \begin{array}{rr} \alpha\text{-}cl(A)=a\text{-}cl(A)\subseteq cl(A)\\ 
a \text{-}cl(A)=X \end{array} \right\} \Rightarrow cl(A)=X\end{array}$
\\
$\left. 
\begin{array}{r} 
\Rightarrow cl(A)=X \\  (X,\tau) \text{ is  submaximal}
\end{array}\right\}\Rightarrow  \begin{array}{rr}  \\  \!\!\!\!\!\!\!\!\!\!\!\!\!\!\!\!
\left. \begin{array} {rr}  A\in \tau \\ (X,\tau) \text{ is regular} \end{array} \right\} \Rightarrow A \in \delta O(X) \subseteq aO(X)
\end{array}$
\\
$\begin{array}{l}
\Rightarrow A\in aO(X).
\end{array}$
\\

$(b) \Rightarrow (c):$ Let $A \subseteq X$ and $e \text{-}cl(A)=X$.
\\
$\left.\begin{array}{r} A\subseteq X \Rightarrow e \text{-}cl(A) \subseteq a \text{-}cl(A) \\ e \text{-}cl(A)=X \end{array} \right\}\Rightarrow \begin{array}{rr} \\ \!\!\!\!\!\!\!\!\!\!\!\!\!\!\!\! \left. \begin{array}{rr} a\text{-}cl(A)= X\\ 
(X,\tau) \text{ is }  a\text{-submaximal} \end{array} \right\} \Rightarrow A \in aO(X) \subseteq eO(X) \end{array}$
\\
$\begin{array}{l}
\Rightarrow A\in eO(X).
\end{array}$
\\

$(c) \Rightarrow (a):$ Let $A \subseteq X$ and $cl(A)=X$.\\
$\left. 
\begin{array}{r} 
A \subseteq X \\ 
(X,\tau) \text{ is } e\text{-space}
\end{array}\right\} \Rightarrow \!\!\!\!\!\!  \begin{array}{c} 
\mbox{} \\ 
 \left. 
\begin{array}{r} 
e \text{-}cl(A)=cl(A)
\\ 
cl(A)=X
\end{array} 
\right\} \Rightarrow e \text{-}cl(A)=X 
\end{array}$
\\
$\left. 
\begin{array}{r} 
\Rightarrow e\text{-}cl(A) =X\\ 
(X,\tau) \text{ is } e\text{-submaximal}
\end{array}\right\} \Rightarrow \!\!\!\!\!\!\!\!\!\!\!\!\!\! \begin{array}{c} 
\mbox{} \\ 
 \left. 
\begin{array}{r} 
A \in eO(X)
\\ 
(X,\tau) \text{ is } e\text{-space}
\end{array} 
\right\} \Rightarrow A \in \tau. \hspace{4,3cm} \qedhere 
\end{array}$    
\end{proof}	     
	       
\begin{definition}
	        Let $A$ and $B$ be two subsets of a space $X.$ Then, $A$ and $B$ are said to be $a$-separated if $A \cap a\text{-}cl(B) = \emptyset$ and $B \cap a\text{-}cl(A) = \emptyset$.
\end{definition}

\begin{theorem}
Let $A$ and $B$ be two $a$-locally closed sets in a space $X$. If $A$ and $B$ are $a$-separated, then $A \cup B \in aLC(X)$.
\end{theorem}
	      
       \begin{proof}
       Let $A$ and $B$ be $a$-separated and $A$ and $B$ two $a$-locally closed sets.
	      $\left.\begin{array}{rr} A \in aLC(X) \Rightarrow (\exists P \in aO(X))(A=P\cap a\text{-}cl(A)) \\ B \in aLC(X) \Rightarrow (\exists Q \in aO(X))(B= Q \cap a\text{-}cl(B)) \\ (U:= P \cap (X\setminus a\text{-}cl(B)))(V:= Q \cap (X \setminus a\text{-}cl(A))) \end{array}\right\}\Rightarrow$
	      \\
 $\begin{array}{l}
	      \Rightarrow (U,V \in aO(X))(U \cap a\text{-}cl(A)=A)(V \cap a\text{-}cl(B)=B)
	      \end{array}$
	      \\
	      $\begin{array}{ll}((U \cup V) \cap a\text{-}cl(A \cup B) = (U \cup V) \cap a\text{-}cl(A) \cup a\text{-}cl(B))
	      \\
	          \Rightarrow  (U,V \in aO(X))(U \cap a\text{-}cl(A)=A)(V \cap a\text{-}cl(B)=B) \\ (U \cup V) \cap a\text{-}cl(A \cup B)= (U \cap a\text{-}cl(A)) \cup (U \cap a\text{-}cl(B)) \cup (V \cap a\text{-}cl(A)) \cup (V \cap a\text{-}cl(B))\end{array}$
\\
 $\begin{array}{ll}\Rightarrow  (U,V \in aO(X))(U \cap a\text{-}cl(A)=A)(V \cap a\text{-}cl(B)=B)
\\
(U \cup V) \cap a\text{-}cl(A \cup B)= (\underset{A}{\underbrace{U \cap a\text{-}cl(A)}}) \cup (\underset{\emptyset}{\underbrace{U \cap a\text{-}cl(B)}}) \cup (\underset{\emptyset}{\underbrace{V \cap a\text{-}cl(A)}}) \cup (\underset{B}{\underbrace{V \cap a\text{-}cl(B)}})\end{array}$
       \\  
        $\begin{array}{l}\Rightarrow  (U\cup V \in aO(X))(a\text{-}cl(A\cup B)\in aC(X))((U \cup V) \cap a\text{-}cl(A \cup B)=A\cup B)\end{array}$
       \\
	      $\begin{array}{l}
	       \Rightarrow A \cup B \in aLC(X).
	      \end{array}$
	      \end{proof}

	      \begin{lemma}\label{bilge}
        Let $X$ and $Y$ be two topological spaces and $A\subseteq X$ and $B\subseteq Y.$\\
	      $a)$ If $A\in aO(X)$ and $B \in aO(Y),$ then $A \times B \in aO(X \times Y),$
	      \\
	      $b)$ If $A\in aC(X)$ and $B \in aC(Y),$ then $A\times B\in aC(X\times Y).$
	      \end{lemma}
	      
	      \begin{proof}
$a)$ Let $A\in aO(X)$ and $B\in aO(Y)$.\\
	      $
	      \left.\begin{array}{rr}
	      A\in aO(X)\Rightarrow A\subseteq int(cl(\delta\text{-}int(A))) \\ B\in aO(Y)\Rightarrow B\subseteq int(cl(\delta\text{-}int(B))) 
	      \end{array}\right\}\Rightarrow
	      $
	      \\
	      $\begin{array}{rcl}
	      \Rightarrow A\times B & \subseteq & int(cl(\delta\text{-}int(A)))\times int(cl(\delta\text{-}int(B))) \\ & = & int[cl(\delta\text{-}int(A))\times cl(\delta\text{-}int(B))] \\ & = & int(cl[\delta\text{-}int(A)\times \delta\text{-}int(B)])
	      \\ & = & int(cl(\delta\text{-}int(A\times B)))
	      \end{array}$

      Then, we have $A\times B\in aO(X\times Y).$
\\
       
$b)$ Let $A\in aC(X)$ and $B\in aC(Y)$.\\
	      $
	      \left.\begin{array}{rr}
	      A\in aC(X)\Rightarrow A\supseteq cl(int(\delta\text{-}cl(A))) \\ B\in aC(Y)\Rightarrow B\supseteq cl(int(\delta\text{-}cl(B))) 
	      \end{array}\right\}\Rightarrow
	      $
	      \\
	      $\begin{array}{rcl}
	      \Rightarrow A\times B & \supseteq & cl(int(\delta\text{-}cl(A)))\times cl(int(\delta\text{-}cl(B))) \\ & = & cl[int(\delta\text{-}cl(A))\times int(\delta\text{-}cl(B))] \\ & = & cl(int[\delta\text{-}cl(A)\times \delta\text{-}cl(B)])
	      \\ & = & cl(int(\delta\text{-}cl(A\times B)))
	      \end{array}$
      
      Then, we have $A\times B\in aC(X\times Y).$
   \end{proof}
	      
	      \begin{theorem}
	      Let $X$ and $Y$ be two topological spaces and $A\subseteq X$ and $B\subseteq Y$. If $A \in aLC(X)$ and $B \in aLC(Y)$, then $A \times B \in aLC(X \times Y).$
          \end{theorem}
	      \begin{proof}
     Let $A\in aLC(X)$ and $B\in aLC(Y).$
 \\       
       $
\left.\begin{array}{rr}
 A\in aLC(X)\Rightarrow (\exists U_1 \in aO(X))(\exists V_1 \in aC(X))(A = U_1 \cap V_1)\\ B\in aLC(X)\Rightarrow (\exists U_2 \in aO(X))(\exists V_2 \in aC(X))(B=U_2 \cap V_2)	      \end{array}\right\}\overset{\text{Lemma } \ref{bilge}}{\Rightarrow}$
	      \\
	$\left.\begin{array}{rr}\Rightarrow (U_1 \times U_2 \in aO(X))(V_2 \times V_2 \in aC(X))(A \times B =(U_1 \cap V_1) \times (U_2 \cap V_2)) \\ (U_1 \cap V_1) \times (U_2 \cap V_2)=(U_1 \times U_2) \cap (V_1 \times V_2) \end{array}\right\}\Rightarrow $  
 \\
$\begin{array}{rr}\Rightarrow (U_1 \times U_2 \in aO(X))(V_2 \times V_2 \in aC(X))(A \times B =(U_1 \times U_2) \cap (V_1 \times V_2))  \end{array}$
 \\
$\begin{array}{l}\Rightarrow A \times B \in aLC(X \times Y).\end{array}$
	      \end{proof}

\section{Conclusion}
In this article, we defined a new type of set, called $a$-locally closed, by utilizing the notion of $a$-open and $a$-closed sets and investigated their fundamental properties. Also, we obtained some characterizations of this new notion. Moreover, we compared the class of sets with the existing ones in the literature. Furthermore, we proved some relationships between this new notion and the other notions that existed in the literature and we also gave several examples. We hope that this paper will stimulate further research on the notion of locally closedness.


\bibliographystyle{amsplain}

\begin{thebibliography}{99}
\bibitem{1}
N. Bourbaki, General topology, Part I. Reading, MA:\emph{ Addison Wesley} (1966).
\bibitem{2}
N. Bourbaki, Elements of mathematics, General Topology, \emph{Part I. Hermann} (1966).
\bibitem{3}
D. Andrijevic, On $b$-open sets, Matematicki Vesnik. 48 (1996), 59-64.
\bibitem{4}
E. Ekici, Some generalizations of almost contra-super-continuity, Filomat. 21 (2) (2007), 31-44.
\bibitem{5}
E. Ekici, A note on $a$-open sets and $e^*$-open sets, Filomat. 22 (2008a), 89-96.
\bibitem{6}
E. Ekici, On $a$-open sets, $A^{*}$-sets and decompositions of continuity and super-continuity, Annales Univ. Sci. Budapest. Eötvös Sect. Math., 51 (2008c), 39-51.
\bibitem{7}
E. Ekici, New forms of contra continuity, Carpathian J. Math., 24(1) (2008b), 37-45.
\bibitem{8}
E. Ekici, On $e$-open sets, $DP^*$-sets and $DP\mathcal{E}^*$-sets and decompositions of continuity, Arab. J. Sci. Eng. 33 (2A) (2008d), 269-281.
\bibitem{9}
M. Ganstercand I.L. Reilly, Locally closed sets and $LC$-continuous functions Internat. J. Math. and Math. Sci., 12(3) (1989), 417-424.
\bibitem{10}
Y. Gnanambal and K. Balachandran, $\beta$-locally closed sets and $\beta$-$LC$-continuous functions, Mem. Fac. Sci, Kochi Univ. (Math.), 19 (1998), 35-44.

\bibitem{101}
B. İzci, On $e$-locally closed and $a$-locally closed sets, MSc Thesis, Muğla Sıtkı Koçman University, (2023).
\bibitem{11}
N. Levine, Semi-open sets and semi-continuity in topological spaces, Amer. Math. Monthly, 70 (1963), 36-41.
\bibitem{12}
A. A. Nasef, On $b$-locally closed sets related topics, Chaos, Solits. Fractals, 12 (2001), 1905-1915.
\bibitem{13}
A. A. Nasef, Feebly locally closed sets and feebly $LC$-continuous functions.
\bibitem{14}
O. Njastad, On some classes of nearly open sets, Pacific J. Math., 15 (1965), 961-970.
\bibitem{15}
M. H. Stone, Applications of the theory of Boolean ring to general topology, Trans. Amer. Math. Soc. 41 (1937), 375-381.
\bibitem{15}
N. V. Velicko, $H$-closed topological spaces, Amer. Math. Soc., 78(2) (1968), 103-118.



\end{thebibliography}

\end{document}